\documentclass[11pt,reqno]{amsart}
\usepackage{amssymb}
\usepackage{amsmath}
 \usepackage[usenames,dvipsnames]{color}
\usepackage[kerning=true]{microtype}
\usepackage[page]{appendix}
\usepackage[all,cmtip]{xy}
\usepackage{pgf,tikz}
\usepackage{fullpage}
\usepackage{amsthm}
\usepackage{bm}
\usepackage[T1]{fontenc}
\usepackage[applemac]{inputenc}
\usepackage[mathcal]{euscript}
\usepackage[colorlinks=true, citecolor=cyan,urlcolor=blue,linktocpage=true]{hyperref}

\definecolor{midgreen}{rgb}{0.,0.7,0.} 
\definecolor{green}{rgb}{0.,0.6,0.} 

\def\bleu{\textcolor{blue}}

\newcommand{\carrejaune}[2]{\filldraw[draw=black,fill=yellow,line width=.15mm] (#1-.5,#2-.5) rectangle (#1+.5,#2+.5);}
\newcommand{\carrevert}[2]{\filldraw[draw=black,fill=midgreen,line width=.15mm] (#1-.5,#2-.5) rectangle (#1+.5,#2+.5);}
\newcommand{\carrebleu}[2]{\filldraw[draw=black,line width=.15mm, fill=blue, inner sep=1pt] (#1-.5,#2-.5) rectangle (#1+.5,#2+.5);}

\DeclareMathOperator{\area}{area}

\newcommand{\A}{\mathcal{A}}
\newcommand{\sbinom}[2]{\textstyle\binom{#1}{#2}}
\newcommand{\coeff}{\bm{c}}

\newcommand{\Cat}{\mathrm{Cat}}
\newcommand{\Cmuk}[2]{\coeff_{#1}^{\langle #2\rangle}}
\newcommand{\coeffprime}{\bm{d}}
\newcommand{\charac}{\raise 2pt\hbox{\large$\chi$}}

\DeclareMathOperator{\des}{des}

\newcommand{\eperp}[1]{(e_{#1}^\perp\otimes \Id)}
\newcommand{\E}{\mathcal{E}}

\newcommand{\Enk}[2]{\E_{#1}^{\langle #2\rangle}}
\newcommand{\F}{\mathcal{F}}

\newcommand{\MacH}{\widetilde{H}}
\DeclareMathOperator{\GL}{GL}
\DeclareMathOperator{\Id}{Id}
\DeclareMathOperator{\length}{\ell}

\newcommand\myatop[2]{\genfrac{}{}{0pt}{}{#1}{#2}}
\newcommand{\M}{\mathcal{M}}

\newcommand{\N}{\mathbb{N}}
\newcommand{\qbinom}[2]{\genfrac{[}{]}{0pt}{}{#1}{#2}_q}
\newcommand{\Rational}{\mathbb{Q}}

\newcommand{\R}{\mathcal{R}}
\newcommand{\Rnk}[2]{\R_{#1}^{\langle #2\rangle}}

\newcommand{\scalar}[2]{{\langle#1,#2 \rangle}}
\newcommand{\scalarstar}[2]{{\langle#1,#2 \rangle}_{\scriptscriptstyle \otimes}}

\renewcommand{\S}{\mathbb{S}}

\newcommand{\Vdm}{\bm{V}}

\newcommand{\auteur}[1]{{\sc #1}}
\newcommand{\titreref}[1]{{\em #1}}
\newcommand{\vol}[1]{{\bf #1}}
\newcommand{\MathReview}[1]{\href{http://www.ams.org/mathscinet-getitem?mr=#1}{MR#1}.}
\renewcommand{\MathReview}[1]{ }

\newcommand{\pref}[1]{{\rm (\ref{#1})}}

\newcommand{\define}[1]{\bleu{\bf{#1}}}

\newtheorem{lemma}{\bleu{Lemma}}
\newtheorem{conjecture}{\bleu{Conjecture}}

\newtheorem{prop}{\bleu{Proposition}}

\numberwithin{equation}{section}
\numberwithin{lemma}{section}
\numberwithin{rmk}{section}
\numberwithin{cor}{section}

\title[Multi]{$(\GL_k\times\S_n)$-Modules of Multivariate Diagonal Harmonics.
}
\author{F.~Bergeron}
\address{\href{http://bergeron.math.uqam.ca}{D\'epartement de Math\'ematiques, Lacim, UQAM.}}
  \email{\href{mailto:bergeron.francois@uqam.ca}{bergeron.francois@uqam.ca}}
  \date{\bleu{\bf \today}. This work was supported by NSERC}

\begin{document}

%%%%%%%%%%%%%%%%%%%%%%%%%%%%%%%%%%

\begin{abstract} 
This is the first in a series of papers in which we describe explicit structural properties of spaces of diagonal rectangular harmonic polynomials in $k$ sets of $n$ variables, both as $\GL_k$-modules and $\S_n$-modules, as well as some of there relations to areas such as Algebraic Combinatorics, Representation Theory, Algebraic Geometry, Knot Theory, and Theoretical Physics. Our global aim is to develop a unifying point of view for several areas of research of the last two decades having to do with Macdonald Polynomials Operator Theory, Diagonal Coinvariant Spaces, Rectangular-Catalan Combinatorics, the Delta-Conjecture, Hilbert Scheme of Points in the Plane, Khovanov-Rozansky Homology of $(m,n)$-Torus links, etc. 
\end{abstract}

\maketitle
 \parskip=0pt
{ \setcounter{tocdepth}{1}\parskip=0pt\footnotesize \tableofcontents}
\parskip=8pt  
\parindent=20pt

%%%%%%%%%%%%%%%%%%%%%%%%%%%%%%%%%%
\section{\bleu{Introduction}}
Initiated at  the beginning of the 1990s, the work of Garsia and Haiman~\cite{garsia_haiman} relating to Macdonald polynomials sparked a strong and long-lasting interest in modules of ``diagonal harmonic polynomials'' in two sets of variables, whose overall dimension is $(n+1)^{n-1}$. Links to parking functions on the combinatorial side, and Hilbert Schemes of points in the plane~\cite{haimanhilb,haimanvanishing} on the Algebraic Geometry one, were soon established. These links have become much better understood in recent years. Since its inception, this line of inquiry has opened many fruitful areas of research, including recent ties with Khovanov-Rozansky homology of $(m,n)$-torus links (see~\cite{haglund_knot,mellit_homology,mellit_braids,morton_samuelson,wilson}), and ties to the study of super-spaces of Bosons-Fermions. A significant part of this story involves a formula for the (bi)graded character of the above-mentioned modules
in the form $\nabla(e_n)$, where $\nabla$ is an operator (introduced in~\cite{ScienceFiction}) having the (combinatorial) Macdonald symmetric functions as joint eigenfunctions, here applied to the degree $n$ elementary symmetric function $e_n$. This operator has many interesting properties on its own, and several important questions about it are still being actively investigated  (see~\cite{multihooks,open,mellit_braids,NegutShuffle}). 

Our aim is to describe several structural properties of the $(\GL_k\times\S_n)$-module, denoted by $\Enk{n}{k}$,  of $k$-variate diagonal harmonic polynomials (previously studied in~\cite{multicomplex}), which specialize at $k=2$ to the module of  diagonal harmonic polynomials of Garsia and Haiman. 
As discussed in~\cite{multicomplex}, there is a filtration of $\S_n$-module (over the field $\Rational$)
  %%%%%%%%%%%%%%%%%%%%%%%%%%%%%%%%%%%%%%%%%%%%%%%%%%
\begin{equation}\label{filtration}
  \xymatrix@-=0.4cm{
           \Rational=\Enk{n}{0}    \ar@{^{(}->}[r] 
	&\Enk{n}{1}  \ar@{^{(}->}[r] 
	&\Enk{n}{2}  \ar@{^{(}->}[r] 
	& \ar@{.}[r]
	&\ \ar@{^{(}->}[r]
	&   \Enk{n}{k}     \ar@{^{(}->}[r] 
	& \ar@{.}[r]
	&\ \ar@{^{(}->}[r]
        & \bleu{\E_{n} } := \displaystyle\lim_{k\rightarrow \infty} \Enk{n}{k},}
\end{equation}
  %%%%%%%%%%%%%%%%%%%%%%%%%%%%%%%%%%%%%%%%%%%%%%%%%%
which ``stabilizes'' when $k$ becomes large enough. 
The first non-trivial module, {\sl i.e.} $\Enk{n}{1}$, of this filtration plays a crucial role in several subjects,  up to natural isomorphisms. Under different guises, it appears as the cohomology ring of the full-flag manifold, as the coinvariant space of the symmetric group $\S_n$, as well as the space of $\S_n$-harmonic polynomials.
These are well-known results going back to the 1950s (see~\cite{sheppard}). This module carries a graded version of the regular representation of $\S_n$,
and  it may simply be described as the linear span of all derivatives (to all orders) of the {Vandermonde determinant} $\Vdm(x)$ in the variables $x=(x_1,\ldots,x_n)$. Among several interesting bases, one of the simplest is the set of partial derivatives $\partial {x}^{\bm{d}}\Vdm(x)$, for
$\bm{d}=(d_1,\ldots, d_n)\in\N^n$ such that $0\leq d_k\leq k-1$, 
where $\partial {x}^{\bm{d}}:= {\partial x_1^{d_1}}\cdots  {\partial x_n^{d_n}}$. 
From this, it may be seen that its graded dimension is the $q$-analogue of $n!$. In fact, the corresponding graded irreducible decomposition as a $\S_n$-module is entirely encoded in terms of the Macdonald polynomial $\MacH_n(q;\bm{z})$, in variables $\bm{z}=(z_i)_{i\in\N}$ (discussed further in the sequel).  

The work of Garsia-Haiman, on the case $k=2$, is seminal in this whole line of investigation. 
The corresponding space, denoted by $\Enk{n}{2}$, may be obtained from $\Enk{n}{1}$ by further closing with respect to \define{higher polarization} operators:
 $\sum_{i=1}^n y_i^j {\partial x_i}$, for any $j\geq 1$ (as well as similar operators obtained by exchanging the role of $x$ and $y$).
Since the early 1990s, many new lines of investigation have been added to their original framework. Noteworthy among recent developments are the successive appearances: of ``rational'', and ``rectangular'' Catalan combinatorics (see~\cite{ALW, reiner,armstrong,open}); as well as interesting ties between these combinatorial settings and the elliptic Hall algebra, introduced by Burban-Schiffmann (see~\cite{BurbanSchiffmann}). This last context offers a broad extension to the spectrum of operators first considered in~\cite{identity}, particularly those that appear in Theorem 4.4 therein.  

In parallel, an explicit study of $\Enk{n}{3}$ (obtained by closing $\Enk{n}{2}$  with respect to polarization operators involving a third set of variables) was started about 10 years ago. Most fundamental questions about it remain open, but its study has suggested new (hard to prove) combinatorial identities linked to the Tamari Lattice as discussed in~\cite{trivariate,hopf}. The general $k$-framework is also considered in a previous paper~\cite{multicomplex}, where some broad properties of the modules of $k$-variate diagonal coinvariant (for any finite complex reflection groups), considering the inductive limit
$ \E_{n}:=\lim_{k\rightarrow \infty} \Enk{n}{k}$ 
as a $\GL_\infty\times \S_n$-module (with commuting actions), and its decomposition:
  \begin{equation}
      \E_{n} \simeq\bigoplus_{\mu\vdash n} \bigoplus_{\lambda} \left(\mathcal{W}_\lambda \otimes \mathcal{V}_\mu\right)^{\oplus c_{\lambda,\mu}},
   \end{equation}
into $\GL_\infty$-polynomial-irreducibles (the $\mathcal{W}_\lambda$'s), and $\S_n$-irreducibles ($\mathcal{V}_\mu$'s).
A crucial missing part in our previous work on the general case was an explicit description of this irreducible decomposition of $\E_n$ (the values of the $c_{\lambda,\mu}$).  We are currently going to describe the structure of this decomposition, as see that it is characterized by a strikingly small set of data. This involves a precise and explicit link between various $\S_n$-isotypic components of the modules under study, established in part via Vertex Operators (see~\cite{zabrocki}).   A surprising corollary of our approach is that we can reconstruct the alternating component of $\E_n$ from the only knowledge of (hook-shape components)  of the module $\Enk{n}{k}$, with $k=\lfloor \frac{n-1}{2}\rfloor$ (rather than having to calculate up to $k=n-1$).

Many seemingly independent aspects of the theory for $k=2$ are nicely explained and tied together via general properties of the $\E_{n}$, and their decomposition into irreducible components. As we will see most of these ties cannot be explained if one stays in the restricted context of $k=2$. In other terms, it is only by going to the general stable framework that the simplicity of the underlying structure is revealed. 
 In particular, we establish a surprising connection with the ``Delta operators'' $\Delta_{e_a}$ (see~\cite{identity}), which generalize the $\nabla$ operator, shedding new light on an open conjecture of Haglund-Remmel-Wilson (see~\cite{CarlssonMellit,haglund}). Indeed, one of our main conjecture states that
 \begin{conjecture}[Delta-by-Skewing]\label{Delta_Conj} For $n=a+b+1$, we have
\begin{equation}\label{calculation_for_hooks}
    \Delta'_{e_{a}}\,e_n(\bm{z})=(\eperp{b}\,\E_n)(q,t;\bm{z}).
    \end{equation}
   \end{conjecture}
   \noindent  
Conjecture~\ref{Delta_Conj} establishes a connection with the "Superspace" of Bosons-Fermions, the study of which is ongoing. In a nutshell, the generic Frobenius characteristic of the superspace is conjectured to be given by the plethysm $\E_n[\bm{q}-\varepsilon \bm{u};\bm{z}]$ (see the appendix for  plethysms involving ``$\varepsilon$'').

%%%%%%%%%%%%%%%%%%%%%%%%%%%%%%%%%%
%%%%%%%% SECTION DEFINITION OF E_n
%%%%%%%%%%%%%%%%%%%%%%%%%%%%%%%%%%
\section{\bleu{Spaces of multivariate diagonal harmonic polynomials}}\label{spaces}
Let us first set up our context and recall some classical notions (see also appendix).
Let $\bm{x}$ be a $\infty\times n$ matrix of variables:
           $$\bm{x}= \begin{pmatrix}
               x_{11} &  x_{12} &\cdots & x_{1n}\\
	       x_{21} &  x_{22} &\cdots & x_{2n}\\
		\vdots  &\vdots  & \ddots & \vdots
          \end{pmatrix}.$$
We think of each row $(x_{i1} ,x_{i2} ,\ldots , x_{in})$ of $\bm{x}$ as a ``set of $n$ variables''. We sometimes restrict $\bm{x}$ to its first $k$ rows, thus getting a $k\times n$ matrix denoted by $\bm{x}_{\langle k\rangle}$; and then consider the ring $\Rnk{n}{k}:=\Rational[\bm{x}_{\langle k\rangle}]$, of polynomials  in the variables $\bm{x}_{\langle k\rangle}$. Letting $k$ go to infinity, we get the limit ring $\R_n:=\Rational[\bm{x}]$ of polynomials  in the variables $\bm{x}$. To each $\infty\times n$ matrix of positive integers $\alpha=(a_{ij})$, of finite support, we associate the \define{monomial} 
    $$\bm{x}^\alpha:=\textstyle \prod_{ij} x_{ij}^{a_{ij}},$$ 
 in $\R_n$. The \define{$i^{\rm th}$-degree} of  $\bm{x}^\alpha$ is the sum of the entries lying on row $i$ of $\alpha$: 
   $$\deg_i( \bm{x}^\alpha):=a_{i1}+a_{i2}+\ldots+a_{in}.$$
It follows from the constraint on $\alpha$, that the \define{degree} of $\bm{x}^\alpha$, defined as
   $$\deg(\bm{x}^\alpha):=(\deg_{1}(\bm{x}^\alpha),\deg_{2}(\bm{x}^\alpha),\ldots,\deg_{i}(\bm{x}^\alpha),\ldots),$$
lies in the semi-group $\N^\infty$ of \define{finite support sequences of integers}. In this way, one gets
two gradings, one by degree and the other by \define{total degree}, on the ring $\R=\R_n$
       $$\R = \bigoplus_{\bm{d}\in\N^\infty} \R_{\bm{d}},\qquad {\rm and}\qquad \R = \bigoplus_{j\in\N} \R^{(j)},$$
with \define{homogeneous components} $\R_{\bm{d}}$ and  $\R^{(j)}$, respectively  spanned by the monomials $\bm{x}^\alpha$ having \define{row sum sequence} equal to $\bm{d}=(d_1,d_2,\ldots)$, and by monomials of total degree $j=d_1+d_2+\dots=|\bm{d}|$. In other terms,
	$$\R^{(j)}:=\bigoplus_{|\bm{d}|=j} \R_{\bm{d}}.$$
We denote by $\rho(\alpha)$ the row sum sequence $(d_1,d_2,\ldots)$, where $d_i:=a_{i1}+a_{i2}+\ldots+a_{in}$.
The $\R_{\bm{d}}$' are clearly such that $\R_{\bm{d}}\R_{\bm{d}'}\subseteq \R_{\bm{d}+\bm{d}'}$.
For a given $\bm{d}\in\N^\infty$, a polynomial $F(\bm{x})$ lies in $\R_{(\bm{d})}$ if and only if $F(\bm{q}\cdot \bm{x})=\bm{q}^{\bm{d}} F(\bm{x})$,
with $\bm{q}$ standing for a diagonal matrix ${\rm diag}(q_1,q_2,\ldots,\ldots)$ of formal ``parameters'' $q_i$; writing $\bm{q}^{\bm{d}}:=\prod_i q_i^{d_1}$, for any  $\bm{d}$ in $\N^\infty$.

\subsection*{\texorpdfstring{$(\GL_k\times \S_n)$}{GLkSn}-Modules} The ring $\R=\R_n$ (or $\R=\Rnk{n}{k}$) is equipped with two commuting actions, one of $\S_n$ and one of the group\footnote{Recall that, for $k=\infty$, this is the group of $\N\times n$ matrices that coincide with the identity matrix in all rows rows but a finite number.} $\GL_k$, for $k\in \N\cup\{\infty\}$, which may be jointly  defined by
    $$F(\bm{x})\longmapsto F(g\cdot \bm{x}\cdot \sigma),\quad {\rm for}\quad g\in \GL_k \quad {\rm and}\quad \sigma\in \S_n.$$
Here, elements of $\S_n$ are considered as $n\times n$ permutation matrices. In particular, the above turns $\R$ into a polynomial representations of $\GL_k$. The $\S_n$-action preserves the grading, so that each homogenous component $\R_{\bm{d}}$ inherits the $S_n$ action. The $\GL_k$-action preserves total degree, so that $\R^{(j)}$ inherits the whole $(\GL_k \times \S_n)$-action. 

All the modules $\M$ that we are going to consider in the sequel are homogeneous $(\GL_k\times \S_n)$-submodules of $\R$, and polynomial representations of $\GL_k$ (with $k$  in $\N\cup\{\infty\}$). They thus afford a decomposition, graded by degree, on the form
 $$\M=\bigoplus_{\bm{d}\in\N^k} \M_{\bm{d}},\qquad {\rm with}\qquad \M_{\bm{d}}:=\M\cap \R_{\bm{d}}.$$
Each homogenous component $\M_{\bm{d}}$ affords some basis $\mathcal{B}_{\bm{d}}$, of homogeneous polynomials. The \define{$\GL_k$-character} of $\M$ (it also said to be its \define{Hilbert series}) may then simply calculated as
     $$\M(\bm{q}):=\sum_{\bm{d}\in\N^k} \dim(\M_{\bm{d}})\, \bm{q}^{\bm{d}} .$$
Classical results the representation theory of $\GL_\infty$ ensure that these characters are symmetric function of the variables $\bm{q}=(q_1,q_2,\ldots)$, which expand as positive integer linear combinations of Schur functions $s_\lambda(\bm{q})= s_\lambda(q_1,q_2,\ldots)$, indexed by partition $\lambda$. Recall that positive integral linear combination of Schur functions are said to be \define{Schur positive}.
Also recall that it is usual to encode characters of representations of $\S_n$ as symmetric functions\footnote{This is just to say that we have polynomials in a denumerable set of ``abstract''  variables $\bm{z}=(z_i)_{i\in\N}$.} via their \define{Frobenius transform}. In this encoding, irreducible $\S_n$-modules correspond to Schur functions $s_\mu(\bm{z})$, for $\mu$ partitions of $n$.  Summing up, the $(\GL_\infty\times \S_n)$-modules $\M$ we consider, have decompositions that can be presented in the format
\begin{equation}\label{formule_avec_parametres}
  \M(\bm{q};\bm{z})=\sum_{\mu\vdash n}\sum_\lambda a_{\lambda\mu} s_\lambda(\bm{q}) s_\mu(\bm{z}),
\end{equation}
with the $a_{\lambda\mu}$ standing for multiplicities of $(\GL_\infty\times \S_n)$-irreducibles. Observe that the ``restriction'' to variables in $\bm{X}_k$ can simply be obtained by setting $q_i=0$ for al $i>k$.  We use the following ``Schur-$\otimes$-Schur'' format to express our formulas for characters
\begin{equation}\label{formules_tenseurs}
   \M= \sum_{\mu\vdash n}\sum_\lambda a_{\lambda\mu} s_\lambda \otimes s_\mu,
 \end{equation}
with the tensor product allowing us to distinguish between Schur functions that are characters of $\GL_\infty$ (those on left-hand side), and Schur functions that encode $\S_n$-irreducibles (those on right-hand side). The specialization $\M(\bm{q};\bm{z})$ of $\M$ corresponds to the ``evaluation'' $s_\lambda(\bm{q}) s_\mu(\bm{z})$ of the tensor product $s_\lambda\otimes s_\mu$.

 %%%%%%%%%%%%%%% SUBSECTION %%%%%%%%%%%%%%%%%%%%%%%%

\subsection{Diagonal harmonic polynomials} As allude to in the introduction, the modules that we consider are obtained as follows.
Starting with the classical {Vandermonde determinant} $\Vdm_{n}=\Vdm_{n}(\bm{x})$, in the variables $\bm{x}=(x_1,\ldots,x_n)$ (the first row of $\bm{X}$), we consider the smallest graded $\S_n$-submodule $\Enk{n}{k}$  of $\Rnk{n}{k}$ which contains $\Vdm_n$, and is closed under:
\begin{itemize}
\item partial derivatives with respect to any variables in $\bm{X}_k$ (the first $k$ rows of $\bm{X}$); as well as
\item \define{higher polarization} operators $\sum_{i=1}^n u_i \partial v_i^j$,
 for any pair or rows $(u_i)_i$ and $(v_i)_i$ of $\bm{X}_k$, and any $j\geq 1$.
 \end{itemize}
The fact that $\Enk{n}{k}$ is closed under polarization ensures that it is also a  $\GL_k$-submodule of $\Rnk{n}{k}$. 
The filtration in~\pref{filtration} is compatible with the action $\S_n$; and with the restriction from $\GL_{k}$ to $\GL_{k-1}$, which corresponds to restriction to polynomials in the variables $\bm{X}_{k-1}$ (clearly contained in $\bm{X}_k$). Hence, we get an action of $\GL_\infty\times \S_n$ on $\E_n$, by passing to the associated (stable) limit. It has been shown in~\cite{multicomplex} that one reaches stability at $k=n-1$. This implies that the multiplicities $c_{\lambda\mu}$, in the ``Schur-$\otimes$-Schur'' character $ \E_{n}=\sum_{\mu\vdash n}\sum_\lambda c_{\lambda\mu} s_\lambda \otimes s_\mu$,
are non-vanishing only for partitions $\lambda$ having at most $n-1$ parts, and having size (sum of parts) at most $\binom{n}{2}$. 
Each $\S_n$-isotypic components of type $\mu$  affords the structure of a $\GL_\infty$-module. We denote by $\coeff_{\mu}:=\sum_{\lambda} c_{\lambda\mu}\, s_\lambda$ the corresponding $\GL_\infty$-character, 
and say that this is the \define{coefficient} of $s_\mu(\bm{z})$ in $\E_n(\bm{q};\bm{z})$.
The character of $\Enk{n}{k}$ is readily obtained from $\E_n$ as the evaluation:
  \begin{equation}\label{definition_En}
      \E_n(\bm{q};\bm{z}) =\sum_{\mu\vdash n} \coeff_{\mu}(\bm{q}) s_\mu(\bm{z}),
   \end{equation}
in $k$ parameters $\bm{q}=q_1,\ldots, q_k$ (thus $k$ is specified), and formal variables $\bm{z}=(z_i)_{i\in \N}$. In particular,
\begin{equation}
     \E_n(q;\bm{z})=\frac{h_n[\bm{z}/(1-q)]}{h_n[1/(1-q)]},\qquad {\rm and}\qquad  \E_n(q,t;\bm{z})=\nabla(e_n)(q,t;\bm{z}).
 \end{equation}
For $n\leq 5$, we have the following explicit values for $\E_n$:
\begin{itemize}
\setlength{\itemindent}{-.3in}
\setlength\itemsep{4pt}
{\small
\item[] $\E_0 = 1\otimes 1$,
\item[] $\E_1=1\otimes s_{1}$,
\item[] $\E_2=1\otimes s_{2}+s_{1}\otimes s_{11}$,
\item[] $\E_3=1\otimes s_{3}+(s_{1} + s_{2})\otimes s_{21}+(s_{11} + s_{3})\otimes s_{111}$,
\item[] $\begin{aligned}&\E_4=1\otimes s_{4}+(s_{1} + s_{2} + s_{3})\otimes s_{31}+(s_{21} + s_{2} +  s_{4})\otimes s_{22}\\
             &\qquad +(s_{11} + s_{21} + s_{31} + s_{3}  + s_{4} + s_{5})\otimes s_{211}+(s_{111} + s_{31} + s_{41} + s_{6})\otimes s_{1111},
             \end{aligned}$
\item[] $\begin{aligned}&\E_5=1\otimes s_{5} + (s_{1} + s_{2} + s_{3} + s_{4}) \otimes s_{41}
     %\\ &\qquad\qquad  
         + (s_{22} +   s_{21} +   s_{31} + s_{41}  + s_{2} + s_{3} + s_{4} + s_{5} + s_{6}) \otimes s_{32}\\
         &\qquad  + (s_{32} + s_{11} + s_{21} + 2s_{31} + s_{41} + s_{51} + s_{3} + s_{4} + 2s_{5} + s_{6} + s_{7}) \otimes s_{311} \\
         &\qquad + ( s_{211} + s_{311} + s_{22} + s_{32} + s_{42}
  %       \\ &\qquad\qquad\qquad  
          + s_{21}  + s_{31} + 2s_{41} + 2s_{51} + s_{61} + s_{4}  + s_{5} + s_{6} + s_{7} + s_{8}) \otimes s_{221}\\
         &\qquad  + (s_{111} + s_{211} + s_{311} + s_{411} + s_{33} + s_{32} + s_{42} + s_{52}
         \\  &\qquad\qquad\qquad  
         + s_{31} + 2s_{41} + 2s_{51} + 2s_{61} + s_{6} + s_{7} + s_{71} + s_{8} + s_{9}) \otimes s_{2111}\\
         &\qquad  + (s_{1111} + s_{311} + s_{411} + s_{511}
     %    \\ &\quad\qquad\qquad   
         + s_{43} + s_{42} + s_{62} + s_{61} + s_{71} + s_{81} + s_{(10)}) \otimes s_{11111}.
         \end{aligned}$}
\end{itemize}
We present in a similar  format the expansions of $\nabla(e_n)$ (which corresponds to the restriction of $\E_n$ to the $\lambda$'s that have at most $2$ parts). For sure, these expressions for $\nabla(e_n)$ may be directly calculated from their ``usual'' expansion as linear combination of Schur function with coefficients in $\N[q,t]$, simply by rewriting these coefficients (which are symmetric polynomials in $q$ and $t$) in terms of Schur functions. For instance, on rewrites
   $$\nabla(e_3)(q,t;\bm{z})= s_3(\bm{z})+(q+t+q^2+qt+t^2) s_{21}(\bm{z})+(qt+q^3+q^2t+qt^2+t^3)s_{111}(\bm{z})$$
  as $\nabla(e_3)=1\otimes s_3 +(s_1+s_2)\otimes s_{21}+(s_{11}+s_3)s_{111}$.
This allows us to express $\E_6$ in the following reasonably compact manner, assuming that the Schur-$\otimes$-Schur expansion of $\nabla(e_6)$ is known.
\begin{itemize}
\setlength{\itemindent}{-.3in}
\item[] 
{\small $ \begin{aligned}
&\E_6=\nabla(e_6)+(s_{221} + s_{411})\otimes  s_{33}
	%\\  &\qquad   
         + (s_{221} + 2s_{321} + s_{421} + s_{211} +  2s_{311} + 2s_{411} +  2s_{511} + s_{611})\otimes  s_{321}\\
         &\qquad   + (s_{331} + s_{321} + s_{421} + s_{521} + s_{111} + s_{211} + 2s_{311} + 2s_{411} + 2s_{511} + s_{611} + s_{711})\otimes  s_{3111}\\
         &\qquad    + (s_{3111} + s_{331} + s_{221} + s_{321} + s_{421} + s_{521} + s_{311} + s_{411} + 2s_{511} + s_{611} + s_{711})\otimes  s_{222}\\
         &\qquad    + (s_{2111} + s_{3111} + s_{4111} + s_{331} + s_{431} + s_{221} + 2s_{321} + 3s_{421} + 2s_{521} + s_{621}\\
         &\qquad\qquad\qquad + s_{211} + s_{311} + 3s_{411} + 3s_{511} + 4s_{611} + 2s_{711} + s_{811})\otimes  s_{2211}\\
         &\qquad    + (s_{1111} + s_{2111} + s_{3111} + s_{4111} + s_{5111} + s_{331} + 2s_{431} + s_{531}   + s_{321} + 2s_{421} + 2s_{521}\\
         &\qquad\qquad\qquad  + 2s_{621} + s_{721} + s_{311} + 2s_{411} + 3s_{511} + 3s_{611} + 3s_{711} + 2s_{811} + s_{911})\otimes  s_{21111}\\
         &\qquad    + (s_{11111} + s_{3111} + s_{4111} + s_{5111} + s_{6111} + s_{441} + s_{431} + s_{531} + s_{631}+ s_{421} \\
         &\qquad\qquad\qquad + s_{521} + s_{621} + s_{721} + s_{821} + s_{611} + s_{711} + 2s_{811} + s_{911} + s_{(10,1,1)})\otimes  s_{111111}.
\end{aligned} $}
\end{itemize}

 %%%%%%%%%%%%%%% SUBSECTION %%%%%%%%%%%%%%%%%%%%%%%%

\subsection{Numerical Specializations}
We may evaluate $s_\mu$ at $k\in\N$, using the well-known formula
\begin{equation}\label{evaluation_schur}
   s_\mu(k)=s_\mu(\underbrace{1,1,\ldots,1}_{k\ {\rm copies}},0,\ldots)=\textstyle \prod_{(i,j)\in\mu} \frac{k+(j-i)}{h(i,j)},
  \end{equation}
where $(i,j)$ runs over the set of cells of $\mu$, and $h(i,j)$ stands for the associated \define{hook length}. The stableness  of~\pref{filtration} (discussed in the sequel) implies that the (ungraded) Frobenius characteristic of the $\S_n$-module $\Enk{n}{k}$, obtained as $\E_{n}(k;\bm{z})$, 
is polynomial in the parameter $k$ for any given $n$. For instance, $\E_{2}(k;\bm{z})= k\,s_{11}(\bm{z}) + s_{2}(\bm{z})$,
 and
\begin{align*}
\E_{3}(k;\bm{z})&=\textstyle \big(\binom{k+2}{3} + \binom{k}{2})\,s_{111}(\bm{z})  + \big(\binom{k+1}{2} +  \binom{k}{1}\big)\,s_{21}(\bm{z}) + s_{3}(\bm{z})\\
   &= \textstyle\big(\binom{k+1}{3} + \binom{k-1}{2})\,e_3(\bm{z})  + \big(\binom{k}{2} +  2\binom{k-1}{1}\big)\,\,e_{21}(\bm{z}) + e_{111}(\bm{z});\\
\E_{4}(k;\bm{z})&=\textstyle \big(\binom{k+2}{3} + \binom{k}{2})\,s_{111}(\bm{z})  + \big(\binom{k+1}{2} +  \binom{k}{1}\big)\,s_{21}(\bm{z}) + s_{3}(\bm{z})
\end{align*}
Using the fact that $\dim \Enk{n}{k} :=\scalar{ \E_{n}(k;\bm{z})}{p_1(\bm{z})^n}$ and
 $\dim \A_{n}^{\langle k\rangle } :=\scalar{ \E_{n}(k;\bm{z})}{e_n(\bm{z})}$, we deduce from the above that
the dimensions of $\E_{n}^{\langle k\rangle}$ and $\A_{n}^{\langle k\rangle}$ are also polynomial in $k$, for any given $n$.  For instance, we have the following positive integer linear combinations of binomial coefficient polynomials
\begin{itemize}\itemsep=4pt
\item [] $ \begin{aligned}
&  \dim\,\Enk{2}{k} = k+1, && \dim\,\A_2^{\langle k\rangle} = k;\\
&  \dim\,\Enk{3}{k}= \textstyle \binom{k}{3} + 5 \,  \binom{k}{2} + 5 \,  k + 1, \qquad &&\dim\,\A_3^{\langle k\rangle} =\textstyle \binom{k}{3} + 3 \, \binom{k}{2} +k;
\end{aligned}$
\item [] $\dim\,\Enk{4}{k}= \textstyle\binom{k}{6} + 12 \,  \binom{k}{5} + 51 \,  \binom{k}{4} + 96 \,  \binom{k}{3} + 78 \,  \binom{k}{2} + 23 \,  k +  1$,
\item [] $\dim\,\A_4^{\langle k\rangle} = \textstyle\binom{k}{6} + 9 \,  \binom{k}{5} + 25 \,  \binom{k}{4} + 29 \,  \binom{k}{3} + 12 \,  \binom{k}{2} +  k$.
\end{itemize}
It would be nice to have an explicit combinatorial understanding of these expressions in general. 
As alluded to in the introduction, we have explicit (or conjectured) expressions for values of these various polynomials for all $n$, when $k\leq3$. Namely.
\begin{itemize}\itemsep=4pt
\item for $k=0$, since $\Enk{n}{0}=\Rational$, we have $\E_{n}(0;\bm{z}) =s_n(\bm{z})$;
\item
for $k=1$, the module $\Enk{n}{1}$ is the regular representation of $\S_n$, of dimension $n!$, and we have the well-known formula
\begin{align}
\E_{n}(1;\bm{z}) &=\sum_{\mu\vdash n} f^\mu\,s_\mu(\bm{z})\label{formule_E1n}\\
			&= e_1(\bm{z})^n,
\end{align}
with $f^\mu$ denoting the number of standard Young tableaux of shape $\mu$ (given by the hook-length formula);

\item
For $k=2$, the module $\Enk{n}{2}$ is the module of ``parking functions'', and we have (see~\cite{stanley})
\begin{align}
\E_{n}(2;\bm{z}) &=\sum_{\mu\vdash n}(-1)^{n-\length(\mu)} (n+1)^{\length(\mu)-1} {z_\mu^{-1}}\,p_\mu(\bm{z})\label{formule_E2n}\\
			&= \frac{1}{n+1}e_n[(n+1)\,\bm{z}].
\end{align}
From this one easily derives
\begin{equation}
\dim\,\Enk{n}{2}=(n+1)^{n-1},\qquad{\rm and}\qquad 
\dim\,\A_n^{\langle 2\rangle}=\frac{1}{n+1}\binom{2\,n}{n};
\end{equation}

\item
For $k=3$ the following formula is conjectured to hold (see~\cite{trivariate})
\begin{align}
\E_{n}(3;\bm{z}) &=\sum_{\mu\vdash n} (-1)^{n-\length(\mu)} (n+1)^{\length(\mu)-2}\,{z_\mu^{-1}}\,p_\mu(\bm{z})\,\textstyle \prod_{k \in\mu} \binom{2\,k}{k},\label{formule_E3n}\\
  &=\frac{1}{(n+1)^2}\Phi_n[2(n+1)\,\bm{z}],\quad {\rm with}\quad\Phi_n(\bm{z}) :=\textstyle \sum_{\mu\vdash n}f_\mu(\bm{z})\prod_{k\in \mu} \Cat_k, 
\end{align}
implying that
\begin{equation}
\dim\,\Enk{n}{3}=2^n(n+1)^{n-2},\qquad {\rm and}\qquad 
\dim\,\A_n^{\langle 3\rangle}=\frac{2}{n(n+1)}\binom{4\,n+1}{n-1}.
\end{equation}
\end{itemize}

\noindent
No such simple formulas are known for larger positive values of $k$. It may be seen that, at $k=-1$, the $k$-variate polynomial expression $\E_n(k;\bm{z})$ evaluates to $p_n(\bm{z})$. 
%%%%%%%%%%%%%%% SUBSECTION %%%%%%%%%%%%%%%%%%%%%%%%

\subsection{An intriguing negative evaluation}
We have observed that, at $k=-2$, the $k$-variate polynomial $ \dim(\E_{n}^{\langle k\rangle})$ appears to evaluate to the signed Catalan numbers $(-1)^{n-1}\Cat_{n-1}$. Moreover, this intriguing signed Catalan property seems to afford the refinement:
\begin{equation}\label{formule_catalan}
\E_{n}(-2;\bm{z})=(-1)^{n-1}\sum_{\mu\vdash n} \pi(\mu) \,\Cat_{\length(\mu)-1}\, f_\mu(\bm{z}),
\end{equation}
where $\pi(\mu)$ is the product of the parts of $\mu$. In turn, this is the specialization at $q=1$ of the formula
\begin{equation}
{(-q)^{n-1}\E_{n}[-q-1/q;\bm{z}]=\sum_{\mu\vdash n} \textstyle{\prod_{k\in\mu}([k]_{q^2}) \,C_{\length(\mu)-1}(-q)\, f_\mu(\bm{z})}}
\end{equation}
where 
$$ C_n(q):=\sum_{k=1}^{2n+1}\sbinom{n}{\lfloor{(k-1)/2\rfloor}}\sbinom{n}{\lfloor{k/2\rfloor}} q^{k-1};$$
since we may check that $C_n(-1)=\Cat_{n}$.

%%%%%%%%%%%%%%%%%%%%%%%%%%%%%%%%%%
%%%%%%%% SECTION MAIN
%%%%%%%%%%%%%%%%%%%%%%%%%%%%%%%%%%
\section{\bleu{Main results-conjectures}}
Considering the ``scalar product''
such that $\langle f\otimes s_\nu, s_\mu\rangle = f$, so that $\langle \E_n,s_\mu\rangle$ is the coefficient of $s_\mu$ in $\E_n$, we may express our first main ``fact'' as follows:
\begin{conjecture}[Hook-Components]\label{skewing_theorem} For all $n$ and all $0\leq k \leq n-1$, if $\mu$ is the hook shape $(k+1,1^{n-k-1})$, then we have the identity
   \begin{equation}\label{hook_skew}
       e_k^\perp \A_n = \langle \E_n, s_\mu\rangle.
   \end{equation}
In particular, $e_{n-1}^\perp \A_n=1$. 
\end{conjecture}

One of the interesting implication of~\pref{hook_skew}, together with Conjecture.\ref{length_conjecture} below, is that we can reconstruct $\A_n$ from (very) partial knowledge of the values of 
the $\langle \E_n, s_\mu\rangle$. To see how this goes, let us first state the following conjecture, defining the \define{length} $\ell(f)$ of a symmetric function $f$, to be the maximum number of parts $\ell(\lambda)$ in a partition $\lambda$ that index a Schur function $s_\lambda$ occurring with non-zero coefficient $a_\lambda$ in its Schur expansion $f=\sum_\lambda a_\lambda f_\lambda$. In formula: 
    $$\ell(f) =\max_{a_\lambda\not=0} \ell(\lambda).$$

\begin{conjecture}[Coefficient-Length]\label{length_conjecture} For all partition $\mu$ of $n$, we have $\ell(\langle \E_n, s_\mu\rangle) = n-\mu_1$. In particular, $(n-k-1)$ is the length of the coefficient of $s_\mu$, for the hook-shape $\mu=(k+1,1^{n-k-1})$.
\end{conjecture}

For $k=n-1$, this is compatible with Theorem~\ref{skewing_theorem} which implies that 
   $$\A_n=s_{1^{n-1}}+(\hbox{lower length terms}),$$ 
   so that $\A_n$ is indeed of length $n-1$.
As another example, the length of $\langle \E_n, s_{(n-1,1)}\rangle$ is conjectured to be equal to $1$, so that we get
\begin{equation}\label{coefficient_of_n1}
    \langle \E_n, s_{(n-1,1)}\rangle= \langle \nabla(e_n), s_{(n-1,1)}\rangle = s_1+s_2+\ldots + s_{n-1},
 \end{equation}
  since the second equality is well known.

%%%%%%%%%%%%%%%%%%%%%%
\subsection{\texorpdfstring{$\A_n$}{A}-reconstruction}
Let us illustrate how, assuming the Coefficient-Length Conjecture, we may reconstruct $\A_n$. We already know that in general $\A_n=s_{1^{n-1}}+(\hbox{lower length terms})$, so that
 $e_{n-2}^\perp \A_n=s_1+(\hbox{other length $1$ terms})$.
 In fact, from~\pref{coefficient_of_n1}, we also get for all $n$ that 
   $$e_{n-2}^\perp\, \A_n=s_1+s_2+\ldots + s_{n-1},$$
which forces
\begin{equation}\label{deuxieme_approximation_An}
   \A_n=s_{1^{n-1}}+ \sum_{k=3}^{n} s_{(k,1^{n-3})} +(\hbox{terms of length $<n-2$}).
 \end{equation}
 Likewise, all terms of length $n-3$ of $\A_n$ are imposed by the identity
\begin{equation}\label{troisieme_approximation_An}
   e_{n-3}^\perp \A_n = \langle \E_n,s_{n-2,1,1}\rangle = \langle \nabla(e_n),s_{n-2,1,1}\rangle,
 \end{equation}
  which results from the assumption that $ \langle \E_n,s_{n-2,1,1}\rangle$ is of length $2$, hence it value is entirely determined by $\nabla(e_n)$.
For instance, with $n=6$, this gives 
   \begin{align*}
     \A_6&=s_{11111} + s_{3111} + s_{4111} + s_{5111} + s_{6111}\\ 
         &\qquad\qquad + s_{441} + s_{431} + s_{531} + s_{631}+ s_{421} + s_{521} + s_{621} + s_{721} + s_{821}\\
         &\qquad\qquad    + s_{611} + s_{711} + 2s_{811} + s_{911} + s_{10.11}+(\hbox{terms of length $\leq 2$}).
     \end{align*}
The remaining missing terms are thus readily calculated since they correspond exactly to the Schur expansion of $\langle\nabla(e_6),e_6\rangle$, which we give below for completeness sake.
\begin{align*}
 \langle\nabla(e_6),e_6\rangle &= s_{44} + s_{64}  + s_{74} + s_{63} + s_{73} + s_{83} + s_{93} + s_{72} + s_{82} + s_{92} + s_{10.2} + s_{11.2}\\
  &\qquad\qquad + s_{10.1} + s_{11.1} + s_{12.1} + s_{13.1} + s_{15.}
\end{align*}
Once again, let us underline that for all $n$, the value of $\nabla(e_n)$ fixes all the components of length at most $2$ in $\E_n$. 

\subsection{Hook-components reconstruction}
Now, from this explicit knowledge of $\A_n$, we may calculate all the hook-indexed components of $\E_n$. In particular, we may check that we indeed get back the expression already mentioned for  $\langle\E_6,s_{21111}\rangle$, namely
  \begin{align*}
 \langle\E_6,s_{21111}\rangle &= e_1^\perp \A_n\\
    &=\langle\nabla(e_6),s_{21111}\rangle+(s_{1111} + s_{2111} + s_{3111} + s_{4111} + s_{5111} + s_{331} + 2s_{431} + s_{531} \\
         &\qquad\qquad\qquad\qquad + s_{321} + 2s_{421} + 2s_{521} + 2s_{621} + s_{721} + s_{311} + 2s_{411} \\
         &\qquad\qquad\qquad\qquad + 3s_{511} + 3s_{611} + 3s_{711} + 2s_{811} + s_{911});
\end{align*}
as well as that for $\langle\E_6,s_{3111}\rangle$:
  \begin{align*}
 \langle\E_6,s_{3111}\rangle &= e_2^\perp \A_n\\
    &=\langle\nabla(e_6),s_{3111}\rangle+(s_{331} + s_{321} + s_{421} + s_{521} + s_{111} + s_{211} + 2s_{311} + 2s_{411}\\
         &\qquad\qquad\qquad\qquad  + 2s_{511} + s_{611} + s_{711}).
\end{align*}

%%%%%%%%%%%%%%%%%%%%%%%%%
\subsection{Partial reconstruction of other components} Similarly, using both the Component-Length Conjecture and the Delta-via-Skewing Conjecture (see~\ref{Delta_Conj}), we may partially reconstruct other coefficients of $\E_n$, considering that the expansion of $\Delta'_{e_k}(e_n)$ is known for all $k$ and $n$. Observe that the Component-Length Conjecture directly implies that, for all $n$,  
 \begin{equation}
    \langle \E_n,s_{n-2,2}\rangle = \langle \nabla(e_n),s_{n-2,2}\rangle
 \end{equation}
so that we already have the coefficient of $s_{n-2,2}$ fully characterized, on top of those for all hook-shapes. Since the  Delta-via-Skewing Conjecture states that
the length at most $2$ components of $\langle e_k^\perp \E_n,s_\mu\rangle $ coincide with those of $\langle\Delta'_{e_{n-1-k}} e_n,s_\mu\rangle $ for all $\mu$, it may be used to infer components of the corresponding coefficients. We may also deduce from Conjecture.\ref{Delta_Conj} part of Conjecture.\ref{length_conjecture}. For instance,  since $\Delta'_{e_1} e_n= \Delta_{e_1} e_n -1\otimes e_n$ and we have\footnote{Since it underlines Schur positivity in the parameters $q$ and $t$, this is a ``slightly'' stronger statement than that of~\cite{haglund}, but it is equivalent.} (see~\cite[Prop. 6.1]{haglund}) 
\begin{equation}\label{formula_for_delta_e1}
   \Delta_{e_1}(e_n)=\textstyle\sum_{k=1}^{n} s_{k-1}\otimes e_{n-k}e_{k},
\end{equation}
we deduce that $\langle \Delta'_{e_1} e_n,s_\mu\rangle=0$ for all partition $\mu$ having first part larger than $2$. Hence, Conjecture.\ref{Delta_Conj} implies that $\langle e_{n-3}^\perp \E_n,s_\mu\rangle=0$ when $\mu_1>2$, implying that  $\langle  \E_n,s_\mu\rangle=0$ has length at most $2$ in those cases. However, for $\mu$ such that $\mu_1=2$, Formula~\pref{formula_for_delta_e1} implies that $\langle e_{n-2}^\perp \E_n,s_{(2^k,1^{n-2k})}\rangle$ does not vanish and is of length $1$. Thus we conclude that  
\begin{lemma} 
	Conjecture.\ref{Delta_Conj} implies Conjecture.\ref{length_conjecture}, for any partition $\mu$ such that $\mu_1=2$.
\end{lemma}
Moreover, using \pref{formula_for_delta_e1}, Conjecture.\ref{Delta_Conj} states that
  \begin{displaymath}
     \langle e_{n-2}^\perp \E_n,s_{(2^k,1^{n-2k})}\rangle = \sum_{i=k-1}^{n-k-1} s_i.
  \end{displaymath}
Since we already know that $\langle e_{n-2}^\perp \E_n,s_{(2^k,1^{n-2k})}\rangle=0$ if $k\geq 1$, the above identity forces
    \begin{equation}
     \langle  \E_n,s_{(2^k,1^{n-2k})}\rangle = \sum_{i=k-1}^{n-k-1} s_{i+1,1^{n-3}} +(\hbox{terms of length $<n-2$}).
  \end{equation}

 %%%%%%%%%%%%%%%%%%%%%%%%%%%%%%%%%%
%%%%%%%%%%%%%%%%%%%%%%%%%%%%%%%%%%
 \section{\bleu{Structure properties of \texorpdfstring{$\E_{n}$}{f}}} 
 %%%%%%%%%%%%%%% SUBSECTION %%%%%%%%%%%%%%%%%%%%%%%%
\subsection{Hook restriction property}
Multiplicities of hook components in $\A_n$ (the alternating part of $\E_n$), are explicitly known. Indeed, considering in this section that $a$ and $b$ are such that $n=a+b+1$, using~\pref{e_skewing} we may calculate that
\begin{equation}\label{hook_skew}
    \frac{1}{q+u}\sum_{b\geq 0} u^b\,(e_b^\perp s_{(a\,|\,b)})(q)=\frac{s_{(a\,|\,b)}[q-\varepsilon u]}{q+u}= q^au^b;
 \end{equation}
and that $s_\mu[q-\varepsilon u] =0$ whenever $\mu$ is not a hook.
Hence, for any symmetric function $f$ in $\Lambda_n$, 
 \begin{equation}\label{generating_hook}
   \frac{f[q-\varepsilon u]}{q+u}=\sum_{a,b} \scalar{f}{s_{(a\,|\, b)}}\, q^au^b
\end{equation}
may be considered as a generating function for the multiplicities of hook components in $f$.
In the case of $f=\A_n$, the multiplicities are thus readily calculated in view of the following identity, which follows from~\pref{hook_skew}.
 \begin{prop}
 For all $n$, we have
\begin{equation}\label{formula_hooks_en}
   \A_n[q-\varepsilon u]=(q+u)(q^2+u)\cdots (q^{n-1}+u).
 \end{equation}
% Equivalently, 
%\begin{equation}\label{formula_hooks_en_coeff}
%   \scalar{\A_n}{s_{(a\,|\,b)}} = 
%   \end{equation}
 \end{prop}
\begin{proof}[\bf Proof] 
We first recall identity~\pref{Hn_binom}, which may be formulated as
$$\scalar{\E_n(q;\bm{z})}{s_{(a\,|\,b)}(\bm{z})} =  q^{\binom{b+1}{2}}\qbinom{n-1}{b}.$$
Observe that~\pref{hook_skew} implies that $(e_a^\perp  \A_n)(q)=\scalar{\E_n(q;\bm{z})}{s_{(a\,|\,b)}(\bm{z})}$, hence using~\pref{e_skewing} we calculate that
\begin{align*}
   \A_n[q-\varepsilon u]&=\sum_{a+b=n-1} \scalar{\E_n(q;\bm{z})} { s_{(a\,|\,b)}(\bm{z})}\,u^a\\
         &=\sum_{a+b=n-1} q^{\binom{b+1}{2}}\qbinom{n-1}{b} u^a\\
         &=(q+u)(q^2+u)\cdots (q^{n-1}+u).
 \end{align*}
\end{proof}
As another indication of the appropriateness of formula~\pref{formula_hooks_en}, we may directly check that its specialization at $u=0$ does indeed correspond to the known expression for the hook components of $\Enk{n}{1}$. Indeed, the expression
\begin{align*}
\sum_{k=0}^{n-1} \coeff_{(n-k,1^k)}[q]\,z^k
	&= \sum_{k=0}^{n-1} q^{\binom{k+1}{2}} \qbinom{n-1}{k}\,z^k,\\
	&= \textstyle \prod_{i=0}^{n-1} (1+q^i\,z),
\end{align*}
may easily be seen to follow from formula~\pref{Hn_binom}, since $\E_n(q;\bm{z})=\MacH_n(q;\bm{z})$. 
 
 %%%%%%%%%%%%%%%%%%%%%%%%%%%
%%%%%%%%%%%%%%% SUBSECTION %%%%%%%%%%%%%%%%%%%%%%%%

\subsection{Conjectured formula for the hook part multiplicities in other hook components} Our calculations suggest that there is a simple elegant expression for the multiplicity $ \scalar{ s_\lambda}{\coeff_\mu} $ of  $s_\lambda\otimes s_\mu$ in $\E_{n}$, when both $\lambda=(i\,|\,j)$ and $\mu=(a\,|\,b)$ are hook shapes. Indeed, we 
conjecture the following formula for the corresponding generating function via the encoding~\pref{generating_hook}.
\begin{conjecture}\label{Hook_Conj}
For all $a+b+1 =n$, we have
\begin{equation}\label{formula_hooks_n}
       \frac{ \coeff_{(a\,|\,b)}[q-\varepsilon u]}{q+u}= \qbinom{n-1}{b}\,  (q^2+u)\cdots (q^b+u),
    \end{equation}
  using the Gaussian $q$-binomial notation.
    \end{conjecture}
\noindent For example, the size collected restriction to hook shapes of $\langle \E_5,s_{(1\,|\,3)}\rangle$ is
\begin{align*}
  \coeff_{(1\,|\,3)}\big|_{\rm hooks} &=s_{111} + (s_{31}+ s_{211} )  + (2s_{41} + s_{311}) + (s_{6}  + 2s_{51}+ s_{411})\\
     &\qquad  + (s_{7}  + 2s_{61}) + (s_{8}  + s_{71}) + s_{9},
 \end{align*}
and the corresponding polynomial is
\begin{align*}
  \frac{ \coeff_{(1\,|\,3)}[q-\varepsilon u]}{q+u} 
  &=    \qbinom{4}{3}\,  (q^2+u) (q^3+u) \\[4pt]
   &=  (q+1)(q^2+1) (q^2+u)(q^3+u) \\[4pt]
   &=  u^{2}+ (q^{2} u+ qu^{2} ) + (2 \,q^{3} u  +  q^{3} u^{2}) + (q^{5} + 2\,q^{4} u +   q^{3} u^{2})\\
   &\qquad    + (q^{6} +2\, q^{5} u )+ (q^{7} + q^{6}u)     + q^{8} .
 \end{align*}

  %%%%%%%%%%%%%%%%%%%%%%%%%%%
%%%%%%%%%%%%%%% SUBSECTION %%%%%%%%%%%%%%%%%%%%%%%%
\subsection{Length components}\label{section_length}
Define the \define{degree} of 
   $$\sum_{\lambda,\mu} c_{\lambda,\mu} s_\lambda\otimes s_\mu$$ 
 to be the maximum of the values $|\lambda|$, for which $c_{\lambda,\mu}\not=0$.
 The \define{length $d$ component} of $\E_{n}$ is set to be
\begin{equation}
     \E_{n}^{(d)}:=\sum_{\mu\vdash n} \coeff_{\mu}^{(d)}\otimes s_\mu,\qquad {\rm with}\qquad 
     	\coeff_{\mu}^{(d)}=\sum_{\ell(\lambda)=d} c_{\lambda,\mu} s_\lambda.
 \end{equation}
We clearly have
  \begin{equation}
     \E_{n}:=\E_{n}^{(0)}+\E_{n}^{(1)}+\ldots \E_{n}^{(l)},
 \end{equation}
 where $l=\ell(\E_{n})$ is the maximal length occurring in terms of $\E_{n}$. With this notation, Conjecture.\ref{length_conjecture} states that $\ell(\coeff_{\mu})=|\mu|-\mu_1$, and it may be shown that
\begin{equation}\label{degre_et_longueur}
   \deg(\coeff_\mu^{(j)})=\binom{n}{2}-\binom{j}{2}-\sum_{i\in\mu} \binom{i}{2}.
\end{equation}
Let $\Enk{n}{k}$ (similarly for $\Cmuk{\mu}{k}$) stand for the restriction of $\E_{n}$  to its components of length at most $k$. In other terms,  
\begin{equation}
     \Enk{n}{k}:=\E_{n}^{(0)}+\E_{n}^{(1)}+\ldots \E_{n}^{(k)}.
 \end{equation}
For example, 
\begin{align*}
\E_{4}^{(0)}&=1 \otimes s_{4}, \\
\E_{4}^{(1)}&= (s_{1} + s_{2}+ s_{3}) \otimes s_{31} + (s_{2} + s_{4}) \otimes s_{22} + (  s_{3} + s_{4} + s_{5} )\otimes s_{211}  +  s_{6}) \otimes s_{1111,}\\
\E_{4}^{(2)}&= s_{21}\otimes s_{22} + (s_{11} + s_{21} + s_{31} )\otimes s_{211}  + (s_{31}    + s_{41}) \otimes s_{1111},\\
\E_{4}^{(3)}&=s_{111}\otimes s_{1111}.
\end{align*}
Simple general values for length components are:
  \begin{equation}
    \E_{n}^{(0)}=1\otimes s_n,\qquad {\rm and}\qquad \E_{n}^{(n-1)}=e_{n-1}\otimes e_n.
   \end{equation}
The following \define{reduced-length components} can be efficiently used to reconstruct  (part of)  $\E_n$. We set
\begin{equation}\label{reduced_length}
   \varepsilon_n^{(k)}:=\eperp{k}\,\E_n^{(k)},\qquad {\rm (and}\qquad \alpha_n^{(k)}:=e_k^\perp\A_n^{(k)}).
 \end{equation}
For example,
 \begin{align*}
\varepsilon^{(0)}_{4}&= 1 \otimes s_{4},\\
\varepsilon^{(1)}_{4}&=(1 + s_{1} + s_{2})\otimes s_{31}  + (s_{1}+ s_{3}) \otimes s_{22} + (s_{2} + s_{3} + s_{4}) \otimes s_{211} + s_{5} \otimes s_{1111},\\
\varepsilon^{(2)}_{4}&= (1  + s_{1} + s_{2}) \otimes s_{211}  + s_{1} \otimes s_{22} + (s_{2} + s_{3}) \otimes s_{1111},\\
\varepsilon^{(3)}_{4}&=1 \otimes s_{1111}.
\end{align*}

\subsection{Description of  hook components in terms of  \texorpdfstring{$\A_n$}{C}}
From~\pref{hook_skew} we may calculate calculate all the hook components $\coeff_\mu=\langle \E_n,s_{\mu} \rangle$ directly from the alternant component $\A_n=\langle \E_n,s_{1^n} \rangle$, both considered as $\GL_\infty$-character of isotypic components of $\E_n$.
We will see that this implies that we can reconstruct $\A_n$ from much less information than is apparently needed at {\sl prima facie}. 
From now on, let us use \define{Frobenius's notation} for hook shape partitions, writing  $(a\,|\,b)$ for the partition $(a+1,1^b)$ of
$n=a+b+1$ (see Figure~\ref{Fig1}). It is often said that $a$ stands for the \define{arm} of the hook, while $b$ stands for its \define{leg}. We use the Cartesian (aka French) convention to draw diagrams, so that the leg goes up.
 \begin{figure}[ht]
\begin{center}
 \begin{tikzpicture}[thick,scale=.4]
 \node at (-1.2,2.5) {$b\left\{\rule{0cm}{24pt}\right.$};
 \node at (4.5,1.3) {$\overbrace{\hskip3.1cm}^{\textstyle a}$};
\carrevert{0}{4}
\carrevert{0}{3}
\carrevert{0}{2}
\carrevert{0}{1}
\carrebleu{0}{0}
\carrejaune{1}{0}\carrejaune{2}{0}\carrejaune{3}{0}\carrejaune{4}{0}\carrejaune{5}{0}\carrejaune{6}{0}\carrejaune{7}{0}\carrejaune{8}{0}
\end{tikzpicture}
\end{center}
\qquad  \vskip-15pt
\caption{The hook shape $(a\,|\,b)$.}\label{Fig1}
\label{Fig_hook_shape} 
\end{figure}
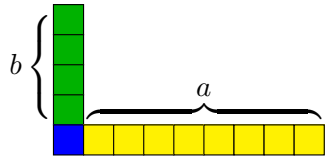

%%%%%%%%%%%%%%%%%%%%%%%%%%%%%%%%%%%%%%%%%%%%%%
%%%%%%%%%%%%%%%%%%%%%%%%%%%%%%%%%%%%%%%%%%%%%%

 \section{\bleu{Link to Delta operators, and the Delta-conjecture}}

\subsection{Skewing versus \texorpdfstring{$\Delta$}{D}-operators}  The ``Delta-Conjecture'' of~\cite{haglund} (see~\pref{delta_conjecture}) proposes an explicit combinatorial description for the evaluation at $e_n$ of the Macdonald ``eigenoperator'' $\Delta'_{e_k}$. The  operators $\Delta'_{e_k}$ were originally introduced in~\cite{identity}, up to a slight shift in eigenvalues.   Our discussion below leads to the conclusion that, for each $\mu$, the coefficient $\scalarstar{ \Delta'_{e_k}(e_n)}{s_\mu}$ should be Schur positive. Hence the same clearly holds true for $\Delta_{e_k}(e_n)$. This is a slightly stronger positivity statement than that of the original Delta-Conjecture. The special case $\Delta_{e_2}(e_n)$ is considered in~\cite{Qiu}, not referring to a Schur expansion but proving an equivalent result. 
 
Conjecture~\ref{Delta_Conj} states that $\Delta'_{e_b}(e_n)$ may be obtained simply by skewing first components in $\E_{n}$ by the elementary symmetric function $e_a$, for $a+b+1=n$. Hence, we may consider  $\eperp{a}\,\E_n$ as a multivariate analogue of $\Delta'_{e_b}(e_n)$. 
Once again we emphasize that the above conjecture could not have been stated in the original restricted context of two sets of variables, namely inside $\Rnk{n}{2}$, in which the $\Delta_{e_k}$ operators are usually considered. 
It has recently been shown (see~\cite{wallace}) that Conjecture~\ref{Delta_Conj} implies Conjecture~\ref{Hook_Conj}. 

It also follows, together with~\pref{hook_skew}, that we have the equality
\begin{equation}\label{Delta_hooks}
   \scalarstar{ \Delta'_{e_b} e_n}{s_{(c|\,d)}} =  \scalarstar{ \Delta'_{e_d} e_n}{s_{(a\,|\,b)}},
\end{equation}
for all $a+b=c+d=n-1$. This follows directly using known identities involving Macdonald polynomials, thus giving indirect support to~\pref{calculation_for_hooks}. Keeping the same convention for $j$ and $k$, we have
 \begin{equation}
(\eperp{j}\,\mathcal{E}_n)(q,1/q;\mathbf{x}) = \sum_{j=0}^k (-1)^{k-j} \frac{e_{j}[[n]_q]}{q^{j\,(n-1)}} \,\frac{e_n[[j+1]_q\,\mathbf{x}]}{[j+1]_q}
\end{equation}

%%%%%%%%%%%%%%%%
%%%%%%%%%%%%%%% SUBSECTION %%%%%%%%%%%%%%%%%%%%%%%%

\subsection{Inferred components of \texorpdfstring{$\E_n$}{E}} Just as before, we consider the decomposition of $\E_n$ into its length components $ \E_n^{(k)}$,
 and the associate reduced-length  components $\varepsilon_n^{(j)}$ (see~\pref{reduced_length}). 
 One may see that the length of $\varepsilon_n^{(j)}$ (that is the maximal length of one of its $\GL_\infty$-coefficients) is equal to $\min(j,n-1-j)$.

Together with this length bound, Conjecture~\ref{Delta_Conj} implies that we may calculate the various $\varepsilon_n^{(j)}$, for all $0\leq j<n\leq 6$, just from the knowledge of the length $2$ expressions $\Delta'_{e_k} e_n$. As a matter of fact, for all $n$, 
we have the general formulas:
\begin{align}
    &\varepsilon_n^{(0)}=1 \otimes s_{n},\\ 
    &  \varepsilon_n^{(1)}=\eperp{1}\,\MacH_n, \\
    &  \varepsilon_n^{(2)}=\eperp{2}\,\nabla(e_n), \\
    &\textstyle   \varepsilon_n^{(n-3)}=\Delta'_{e_{2}}(e_n)  +(s_1+s_2)\otimes e_n +(1\otimes e_{n-1}e_1)- \sum_{k=1}^n s_{k-1}\otimes e_k e_{n-k}.\\
    &\textstyle \varepsilon_n^{(n-2)}=\sum_{k=1}^n s_{k-1}\otimes\,e_{k}e_{n-k} -(1+s_1)\otimes e_n\\
      &\varepsilon_n^{(n-1)}= 1\otimes e_n.
%      \varepsilon_n^{(n-3)}=\Delta'_{e_{2}}(e_n)  +s_2\otimes e_n - \ell_{\leq 2}\big((s_1\otimes \Id)\, \Delta'_{e_{1}}(e_n)\big).
\end{align}

%%%%%%%%%%%%%%% SUBSECTION %%%%%%%%%%%%%%%%%%%%%%%%

\subsection{Partial reconstruction of other components} Using both the Component-Length Conjecture and the Delta-by-Skewing Conjecture (see~\ref{Delta_Conj}), we may partially reconstruct coefficients $\coeff_\mu$ (of $\E_n$) for $\mu$'s that are not hooks, considering that the expansion of $\Delta'_{e_k}(e_n)$ is known for all $k$ and $n$. Observe that the Component-Length Conjecture directly implies that, for all $n$,  
 \begin{equation}
    \scalarstar{ \E_n}{s_{n-2,2}} = \scalarstar{ \nabla(e_n)}{s_{n-2,2}},
 \end{equation}
so that we already have the coefficient of $s_{n-2,2}$ fully characterized, on top of those for all hook-shapes. Since the  Delta-by-Skewing Conjecture states that
the length at most $2$ components of $\scalarstar{ e_k^\perp \E_n}{s_\mu} $ coincide with those of $\scalarstar{\Delta'_{e_{n-1-k}} e_n}{s_\mu} $ for all $\mu$, it may be used to infer components of the corresponding coefficients. We may also deduce from Conjecture~\ref{Delta_Conj} part of Conjecture~\ref{length_conjecture}. For instance,  since $\Delta'_{e_1} e_n= \Delta_{e_1} e_n -1\otimes e_n$ and we have (see~\cite[Prop. 6.1]{haglund}) 
\begin{equation}\label{formula_for_delta_e1}
   \Delta_{e_1}(e_n)=\textstyle\sum_{k=1}^{n} s_{k-1}\otimes e_{n-k}e_{k},
\end{equation}
we deduce that $\scalarstar{ \Delta'_{e_1} e_n}{s_\mu}=0$ for all partition $\mu$ having first part larger than $2$. Hence, Conjecture~\ref{Delta_Conj} implies that $\scalarstar{ e_{n-3}^\perp \E_n}{s_\mu}=0$ when $\mu_1>2$, implying that  $\scalarstar{  \E_n}{s_\mu}=0$ has length at most $2$ in those cases. However, for $\mu$ such that $\mu_1=2$, Formula~\pref{formula_for_delta_e1} implies that $\scalarstar{ e_{n-2}^\perp \E_n}{s_{(2^k,1^{n-2k})}}$ does not vanish and is of length $1$. Thus we conclude that  
\begin{lemma} 
	Conjecture~\ref{Delta_Conj} implies Conjecture~\ref{length_conjecture}, for any partition $\mu$ such that $\mu_1=2$.
\end{lemma}
Moreover, using \pref{formula_for_delta_e1}, Conjecture~\ref{Delta_Conj} states that
  \begin{displaymath}
     \scalarstar{ e_{n-2}^\perp \E_n}{s_{(2^k,1^{n-2k})}} = \sum_{i=k-1}^{n-k-1} s_i.
  \end{displaymath}
Since we already know that $\scalarstar{ e_{n-2}^\perp \E_n}{s_{(2^k,1^{n-2k})}}=0$ if $k\geq 1$, the above identity forces
    \begin{equation}
     \scalarstar{  \E_n}{s_{(2^k,1^{n-2k})}} = \sum_{i=k-1}^{n-k-1} s_{i+1,1^{n-3}} +(\hbox{terms of length $<n-2$}).
  \end{equation}

%%%%%%%%%%%%%%%%%%%%%%%%%%%%%%
%%%%%%%%%%%%%%%%%%%%%%%%%%%%%%
\section{\bleu{The \texorpdfstring{$e$}{e}-positivity phenomenon}} As discussed in~\cite{open}, most of the symmetric functions constructed via the elliptic Hall algebra approach, including $\nabla(e_n)$,  seem to become $e$-positive when specialized at $t=1$.
 We discuss here an analogous situation for $\E_n$, which corresponds to a multi-parameter ``lift'' of $\nabla(e_n)$. For this we consider the specialization of any one of the parameters $q_i$ to the value $1$. In our context of infinitely many parameters, this is handily realized by the plethystic evaluation at $1+\bm{q}$ of the $\GL_\infty$-coefficients $\coeff_{\mu}$ of $\E_{n}$. It is noteworthy that this evaluation is an invertible operation.
 
 For the sake of discussion, let us set $ \F_{n}:= \E_{n}[1+\bm{q};\bm{z}]$,
and write
\begin{align}
    \F_{n}&=\sum_{\mu\vdash n}  \coeff_{\mu}[1+\bm{q}]\otimes s_\mu(\bm{z}) =\sum_{\nu\vdash n} \coeffprime_{\nu}\otimes e_\nu(\bm{z}) \label{e_expression},
\end{align}
so that $\coeffprime_{\nu}$ stands for the coefficients of $e_\nu(\bm{z})$ in $\F_{n}$. Using our scalar product convention, we may also write $\coeffprime_{\nu}= \scalarstar{ \F_{n}}{f_\nu}$ (since the forgotten symmetric functions $f_\nu$ are dual to the $e_\nu$). Then, as far as we can check experimentally, all of the $\coeffprime_{\nu}$ are  \define{Schur positive}.
   For instance, we have
 \begin{align*}
     \F_{4}&=
     	1\otimes e_{1111}+
	(3s_{1} + 2s_{2} + s_{3})\otimes e_{211}+
	(s_{11} + s_{2} + s_{21} + s_{4})\otimes e_{22}\\
& \qquad	+
	(2s_{11} + s_{21} + 2s_{3} + s_{31} + s_{4} + s_{5})\otimes e_{31}+
	(s_{111} + s_{31} + s_{41} + s_{6})\otimes e_{4}.
\end{align*}
  The $\coeff_{\nu}$ are related to the $\coeffprime_{\mu}$ by the formula $ \coeff_{\mu}[1+\bm{q}]=\sum_{\nu} K_{\mu'\lambda}\coeffprime_{\lambda}$,
  where  $K_{\mu\lambda}$ stands for  the usual Kostka number. 
There is a close tie between this $e$-positivity phenomenon, our main theorem~\ref{mainthm}, and related conjectures. To see this, we recall that  the coefficient of $e_n$ in the $e$-expansion of $s_\mu$ vanishes for all $\mu$ except hooks. For hooks, we have $\scalar{ s_{(a\,|\,b)}}{f_{(n)}} =(-1)^a$. We may then calculate that
\begin{align*}\coeffprime_{(n)}&= \scalarstar{ \E_{n}[1+\bm{q};\bm{z}]}{f_{(n)}}=\sum_{\mu\vdash n}  \coeff_{\mu}[1+\bm{q}]\,  \scalar{ s_\mu}{f_{(n)}}\\
   &=\Big(\sum_{a=0}^{n-1}  (-1)^a\, \coeff_{(a\,|\,b)}\Big)[1+\bm{q}] =\Big(\sum_{k\geq 0}  (-1)^a\, e_a^\perp \A_{n}\Big)[1+\bm{q}] .  
\end{align*}
If we recall that   $f[\bm{q}-1] =\sum_{a\geq  0} (-1)^a\,e_a^{\perp}f(\bm{q})$,
we may conclude the above calculation to get
  \begin{equation}
     \coeffprime_{(n)} = (\A_{n}[\bm{q}-1])[1+\bm{q}]=\A_{n},
  \end{equation}
which is Schur positive.
To get other similar results, let $\mu$ be any partition of $n$ which is largest in dominance order among those such that $\coeff_{\mu}\not=0$. Then, it is easy to see that 
 $\coeffprime_{(n)}=\coeff_{1^n}$, and $\coeffprime_{\mu'}=\coeff_{\mu}[1+\bm{q}]$. 
 In particular, $\coeffprime_{(1^n)}=1$.
We thus automatically have  Schur positivity in the these cases.
Experiments suggest that whenever   $a\geq 1$, we have
   \begin{equation}
      \coeffprime_{(a\,|\,b)} = \sum_{j=0}^{b}   \coeff_{(j\,|\, a)},
   \end{equation}
implying Schur-positivity of  $\coeffprime_{(a\,|\,b)}$.

\section*{\bleu{Acknowledgments}} Much of this work would not have been achieved without the possibility of perusing sufficiently large expressions resulting from difficult  explicit computations. These calculations were very elegantly realized by Nicolas Thi\'ery using the Sage computer algebra system, together with an inspired used of the right mathematical properties of objects considered, and his special expertise of higher Specht modules. Indeed, direct calculations of the relevant symmetric function expressions rapidly become unfeasible, even with powerful computers. Such explicit calculations must thus rely on an artful combination of high-level computer algebra skills, and a subtle understanding of the mathematical structures involved. 
%The interesting story of how this was achieved is told in~\cite{thiery}.

\section*{\bleu{Appendix}}

\subsection*{Symmetric functions and plethysm}\label{app_sym}
We mainly use Macdonald's notations (see~\cite{macdonald}).
Thus $p_\mu$, $e_\mu$, $f_\mu$, and $h_\mu$ respectfully stand for the \define{power sum}, \define{elementary}, \define{forgotten}, and \define{complete homogeneous} symmetric functions, with indices integer partitions   $\mu=\mu_1\mu_2\cdots\mu_k$ of $n$.
Recall that these are symmetric functions that are all homogeneous of degree $n=|\mu|:=\mu_1+\mu_2+\ldots +\mu_k$. Assuming that they are evaluated in enough variables, they respectively form bases of the homogenous degree $n$ component $\Lambda_n$ of the graded ring of symmetric function $\Lambda$. It is often useful to avoid writing variables, implicitly making the above assumption. 

The \define{length} $\ell(\mu)$ of a partition $\mu$ is the number of non-zero parts (the $\mu_i$) of $\mu$. 
Recall that partitions are often described in terms of their \define{Ferrers diagram} (herein in French notation). 
This diagram is the set of \define{cells} $(i,j)$ in $\N\times \N$ (considered as points in the usual Cartesian coordinates), with $1\leq i -1\leq \mu_{j+1}$, for $ 0\leq j<\ell(\mu)$. We often write $(i,j)\in\mu$ when the context makes this convention clear. The row lengths of this diagram are the parts of $\mu$. 
One writes $\mu\subseteq \lambda$, when the diagram of $\mu$ is contained in the diagram of $\lambda$. The \define{skew shape} $\lambda/\mu$ has diagram equal to the set-difference of the respective diagrams of $\lambda$ and $\mu$.
The \define{conjugate} $\mu'$ of $\mu$ is the partition with diagram equal to the set  $\{(j,i)\ |\ (i,j)\in\mu\}$. The \define{hook length} of a cell $(i,j)$ of $\mu$ is defined as
                $$h(i,j)=h_{ij}=\mu_{j+1}+\mu_{i+1}'-i-j-1.$$
 The \define{descent set}, $\des(\mu)$, of a partition $\mu$, is the set of $i$ such that $\mu_i>\mu_{i+1}$. For a permutation $\sigma$ of $\{1,\ldots, n\}$, its \define{cycle structure} is the partition $\mu$ of $n$ whose parts are the lengths of its cycles. One denotes by $z_\mu$ the integer such that $n!/z_\mu$ is the number of permutations with cycle structure equal to $\mu$. For partitions $\mu$ and $\nu$, the \define{sum} $\mu+\nu$ is the partition whose parts are $\mu_i+\nu_i$.
 
 A \define{tableau} of shape $\mu$ (or $\lambda/\mu$) is simply a function $\tau:\mu\rightarrow \N$. Thus some value $\tau(i,j)\in\N$ is associated to each cell of $\mu$. If these values are increasing along rows of $\mu$, and strictly increasing along columns, one says that the tableau is a \define{semi-standard Young} tableau. If moreover, $\tau$ is bijective with values going from $1$ to $n=|\mu|$, one says that the tableau is \define{standard}.          
The \define{skew-Schur functions} $s_{\lambda/\mu}$ (with $s_\lambda=s_{\lambda/\emptyset}$, that is for $\mu=\emptyset$ the empty partition) may be considered as enumerators of semi-standard tableaux of shape $\lambda/\mu$. In formula, this says that
    $$s_{\lambda/\mu}(\bm{z})=\sum_{\tau:\lambda/\mu\rightarrow \N} \bm{z}_\tau,\qquad {\rm with}\qquad \textstyle \bm{z}_\tau:=\prod_{c\in\lambda/\mu} z_{\tau(c)},$$
   with $\tau$ running over the set of semi-standard tableaux of shape $\lambda/\mu$. Observe that $s_\lambda(z_1,\ldots,z_k)=0$, whenever $k<\ell(\lambda)$. 
 The Schur functions are orthonormal for the usual \define{Hall scalar product} on $\Lambda$, which may be defined by setting $\scalar{ p_\lambda}{p_\mu} :=z_\mu\delta_{\lambda,\mu}$. As usual $\omega$ stands for the linear multiplicative self adjoint operator such that $\omega\,s_\lambda=s_{\lambda'}$. For this scalar product, the forgotten functions are \define{dual} to the elementary, {\sl i.e.} $\scalar{e_\mu}{f_\lambda} =\delta_{\lambda\mu}$; and the Schur functions are self-dual. 
For a given symmetric function $f$, the linear operator $f^\perp$ is  the \define{adjoint} to the linear operation of multiplication by $f$. In formula, $\scalar{  f\cdot g_1}{g_2}=\scalar{  g_1}{f^\perp g_2}$,  for any symmetric functions $g_1$ and $g_2$. The classical \define{(dual) Pieri rule} (see~\cite{macdonald}), states that
  $e^\perp_k s_\mu =\sum_{\lambda\subset \mu} s_\lambda$,
with the indices of the sum running over partitions $\lambda$ that can be obtained from $\mu$ by removing $k$ cells, no two of which lying on the same row. In particular, $ e^\perp_k s_\mu$ vanishes if $\mu$ has less than $k$ rows.  A symmetric function $f$ is said to be \define{Schur positive} when, for all partition $\mu$, the coefficients of its Schur expansion are polynomials (in some variables) with positive integer coefficients. Sums, products, and ``plethysms'' (see below) of Schur-positive functions are also Schur-positive. Similarly, one has the notion of \define{$e$-positivity}. Since all the $e_\lambda$'s are Schur-positive, we have that $e$-positivity implies Schur positivity. Moreover, $e$-positivity is clearly closed under sums and product.

For symmetric function $f$ and $g$, the \define{plethysm} $f\circ g=f[g]$ is a special case of $\lambda$-ring calculations $f[\bm{a}]$, in  which  symmetric function are considered as operators. The following evaluation rules entirely characterize these, assuming that $\alpha$ and $\beta$ are scalars, and that $\bm{a}$ and $\bm{b}$ lie in some suitable ring:
\begin{align}\itemsep=6pt
{\rm{(i)}}\quad &(\alpha f+\beta g)[\bm{a}]=\alpha\, f[\bm{a}]+\beta\, g[\bm{a}],\qquad &{\rm{(ii)}}&\quad (f\cdot g)[\bm{a}]=f[\bm{a}]\cdot g[\bm{a}],\nonumber\\
{\rm{(iii)}}\quad &p_k[\bm{a}\pm \bm{b}]=p_k[\bm{a}]\pm p_k[\bm{b}],                                   &{\rm{(iv)}}&\quad p_k[ \bm{a}\cdot \bm{b}]=p_k[\bm{a}]\cdot p_k[\bm{b}],\nonumber\\
{\rm{(v)}}\quad &p_k[ \bm{a}/\bm{b}]=p_k[\bm{a}]/p_k[\bm{b}],                                              &{\rm{(vi)}}&\quad p_k[p_j]=p_{kj},\label{def_plethysme}\\
{\rm{(vii)}}\quad &p_k[p_j\otimes p_\ell]=p_{kj}\otimes p_{k\ell},              &{\rm{(viii)}}&\quad p_k[\varepsilon] =(-1)^k,\nonumber\\
{\rm{(ix)}}\quad &p_k[x]=x^k,\ \hbox{if}\  x\  \hbox{a ``variable''},                   
&{\rm{(x)}}&\quad p_k[c]=c,\ \hbox{if}\  x\  \hbox{a ``constant''}.\nonumber
\end{align}
\noindent Hence we specify what are variables and what are constants. The first two properties make it clear that any evaluation of the form $f[\bm{a}]$ may be reduced to instances of the form $p_k[\bm{a}]$. We also assume that property (iii) extends to denumerable sums. Property (viii) implies, that $f[\varepsilon \bm{z}] = \omega\, f(\bm{z})$, for all symmetric function $f$. It may readily be shown that, for all symmetric function $f$,
\begin{equation}\label{e_skewing}
    \sum_{k\geq 0} u^k (e_b^\perp f)(q) = f[q-\varepsilon u].
 \end{equation}
See~\cite{bergeron,haglund} for more on plethysm.

%%%%%%%%%%%%%%%%%
%%%%%%%%%%%%%%% SUBSECTION %%%%%%%%%%%%%%%%%%%%%%%%

\subsection*{Macdonald polynomials, and operators}\label{app_mac}
Recall that the set of \define{combinatorial Macdonald polynomials}  
 $\{\MacH_\mu(q,t;\bm{z})\}_{\mu\vdash n}$
  forms a linear basis of the ring $\Lambda(q,t)$, of symmetric functions in the variables $\bm{z}=(z_i)_{i\in\N}$ over the field $\Rational(q,t)$. They are uniquely characterized by the equations 
 \begin{displaymath}
 \begin{array}{llll}
    \mathrm{(i)}\ {\displaystyle  \scalar{ s_\lambda(\bm{z})}{\MacH_\mu[q,t;(1-q)\,\bm{z}]}=0},\qquad{\rm if}\qquad  \lambda\not\succeq\mu,\\[6pt]
    \mathrm{(ii)}\ {\displaystyle   \scalar{ s_\lambda(\bm{z})}{ \MacH_\mu[q,t;(1-t)\,\bm{z}]}=0},\qquad{\rm if}\qquad   \lambda\not\succeq\mu',\ \mathrm{and}\\ [6pt]
    \mathrm{(iii)}\ {\displaystyle   \scalar{ s_n(\bm{z})}{ \MacH_\mu(q,t;\bm{z})}=1},
      \end{array}
 \end{displaymath} 
 involving plethystic notation.
See~\cite[Section 3.5]{haimanhilb} for more details on these dominance order triangularities. It follows from the $n!$-theorem (see~\cite{haimanvanishing}) that the $\MacH_\mu$ are Schur positive.
The operators $\nabla$ and $\Delta'_{e_k}$, introduced in~\cite{identity}, are special instances of \define{Macdonald eigenoperators}. This is to say that they have the  Macdonald polynomials  $\MacH_\mu$ as joint eigenfunctions.  Their respective eigenvalues are
   $$\Delta_{e_k}(\MacH_\mu):= e_k[B_\mu]\, \MacH_\mu,\qquad \Delta'_{e_k}(\MacH_\mu):= e_k[B_\mu-1]\, \MacH_\mu,\qquad {\rm with}\qquad \textstyle B_\mu:=\sum_{(i,j)\in \mu}q^it^j.$$
On homogeneous symmetric functions of degree $n$ (and only for those),   the operator $\nabla$ coincides with $\Delta_{e_{n}}=\Delta'_{e_{n-1}}$.  In other terms, the associated eigenvalues are
   \begin{equation}\label{defn_nmu}
        { T_\mu:=q^{\eta(\mu')}t^{\eta(\mu)}=\textstyle \prod_{(i,j)\in\mu} q^it^j},\qquad {\rm with}\qquad \eta(\mu):=\textstyle \sum_{k} (k-1)\mu_k,
    \end{equation}
for $\mu=\mu_1\cdots \mu_\ell$. 
Among the many interesting formulas pertaining to the $\MacH_\mu$ (equivalent to formulas that may be found in~\cite[Exer. 2, page 362]{macdonald}), we have
\begin{equation}
  { \scalar{ \MacH_\mu}{s_{(a\,|\,b)}}  = e_b[B_\mu-1]},\quad \hbox{for all} \quad a+b+1=n.
\end{equation}
In particular,
\begin{equation}
 \scalar{ \MacH_\mu}{s_{1^n}}  = T_\mu, \qquad {\rm and}\qquad \scalar{ \MacH_n}{s_{(a\,|\,b)}} = q^{\binom{b+1}{2}}\qbinom{n-1}{b}\label{Hn_binom}
\end{equation}
We also have the symmetries
   \begin{equation}
   \MacH_\mu(q,t;\bm{x}) =T_\mu\,\omega \MacH_\mu(1/q,1/t;\bm{x}), \qquad {\rm and}\qquad 
   \MacH_\mu(t,q;\bm{x}) =\MacH_{\mu'}(q,t;\bm{x}).
   \end{equation}

%%%%%%%%%%%%%%%%%%%%%%
%%%%%%%%%%%%%%% SUBSECTION %%%%%%%%%%%%%%%%%%%%%%%%

\subsection*{The \texorpdfstring{$\Delta$}{D}-conjecture}\label{app_Delta}
In~\cite{haglund}, one finds an explicit combinatorial formula conjectured to be equal to $\Delta'_{e_k}e_n$. With our particular point of view, it takes the form
\begin{equation}\label{delta_conjecture}
  \Delta'_{e_k}e_n=\sum_{\mu\subseteq \delta_{nn}} \Big( \sum_{\myatop{J \supseteq \des(\mu)}{\#J=k}}q^{(J,\bm{a})} \Big)\, \mathbb{L}_\mu(t;\bm{z})
\end{equation}
where the indices $J$ (in the inner sum) run over all suitable subsets of $[n]:=\{1,2,\ldots,n-1\}$, and $(J,\bm{a})$ is shorthand for $\sum_{i\in J} a_i$. We here denote by $\mathbb{L}_\mu(t;\bm{z})$ the vertical strip  LLT-polynomial associated to the path $\mu$. See~\cite{dadderio} for more on these, in particular for a proof that $\mathbb{L}(1+t;\bm{z})$ is $e$-positive.
For instance, when $k=0$, the only non-zero term of the outer sum corresponds to $s_{(\emptyset+1^n)/\emptyset} =e_n$, thus agreeing with $\Delta'_{e_0} e_n$.  At the opposite end of the spectrum, when $k=n-1$, there is but one term in the inner sum (since $J$ must be equal to $[n]$) which is clearly equal to $q^{\area(\mu)}$. Observe that, at $t=1$, the above expression simplifies to
\begin{equation}\label{delta_conjecture_un}
  \Delta'_{e_k}e_n\big|_{t=1}=\sum_{\mu\subseteq \delta_{nn}} \Big( \sum_{\myatop{J\supseteq \des(\mu)}{\#J=k} }q^{(J,\bm{a})} \Big)  s_{(\mu+1^n)/\mu}(\bm{z}).
\end{equation}

\smallskip\goodbreak
%%%%%%%%%%%%%%%%%%%%%
%%%%%%%%%%%%%%% SUBSECTION %%%%%%%%%%%%%%%%%%%%%%%%

\noindent{\bf The $e$-expansions of  $\F_{n}$.}

\noindent Partially expressed in terms of the $\A_n$, we have the following values:

{\footnotesize 
\begin{enumerate}\itemsep=3pt\setlength{\itemindent}{-.3in}
\item[]  $\F_{1}=1\otimes  e_{1}$, \qquad $\F_{2}=\A_2\otimes e_{2}+1\otimes  e_{11}$, \qquad  $\F_{3}=\A_3\otimes e_{3}+(2s_{1} + s_{2})\otimes e_{21}+1\otimes  e_{111}$,
\item[]  $\F_{4}=\A_4\otimes e_{4}+ (e_1^\perp \A_4+\A_3)\otimes e_{31}+(s_{11} + s_{21} + s_{2} + s_{4})\otimes e_{22}
+(3s_{1} + 2s_{2} + s_{3})\otimes e_{211}+1\otimes e_{1111}$,
\item[]  $\begin{aligned} &\F_{5}=\A_5\otimes e_{5}+ (e_1^\perp \A_5+\A_4)\otimes e_{41}+(e_2^\perp \A_5+e_1^\perp \A_4+\A_3)\otimes e_{311}\\
&\qquad +(2s_{111} + 2s_{211} + s_{311} + s_{22} + s_{32} + s_{42}  
+ 2s_{21} + 2s_{31} + 3s_{41} + 2s_{51}+ s_{61}+ 2s_{4} + s_{6} + s_{7} + s_{8})\otimes e_{32}\\
&\qquad
+( s_{22} + 3s_{11} + 4s_{21} + 2s_{31} + s_{41} + 3s_{2} +2s_{3} + 2s_{4} + 2s_{5} + s_{6})\otimes e_{221}\\
&\qquad
+(4s_{1} + 3s_{2} + 2s_{3} + s_{4})\otimes e_{2111}
+1\otimes e_{11111},
\end{aligned}
$

\item[] $
\begin{aligned}
&\F_{6}=\A_6\otimes e_{6} +
(e_1^\perp \A_6+\A_5)\otimes e_{51}+(e_2^\perp \A_6+e_1^\perp \A_5+\A_4)\otimes e_{411}+(e_3^\perp \A_6+e_2^\perp \A_5+e_1^\perp \A_4+\A_3)\otimes e_{3111}\\
&\qquad  +
(2s_{1111} + 2s_{2111} + s_{3111} + s_{4111}  + s_{431} + s_{221} + 2s_{321} + 3s_{421} + s_{521} + s_{621} \\
&\qquad\qquad\qquad + 2s_{211} + 2s_{311} + 5s_{411} + 3s_{511} + 4s_{611} + s_{711} + s_{811}
  + s_{44} + s_{54} + 2s_{43} + 2s_{53} + s_{63} + s_{73} \\
&\qquad\qquad\qquad + 2s_{32} + 3s_{42} + 3s_{52} + 3s_{62} + 3s_{72} + s_{82} + s_{92} \\
&\qquad\qquad\qquad + 2s_{41} + 2s_{51} + 3s_{61} + 3s_{71} + 3s_{81} + 3s_{91} + s_{(10,1)} + s_{(11,1)}
+ 2s_{7} + s_{9} + 2s_{(11)} + s_{(13)})\otimes e_{42}\\
& \qquad +
(s_{1111} + s_{2111} + s_{3111} + s_{331} + s_{221} + s_{321} + s_{421} + s_{521}
+ s_{211} + 2s_{311} + 2s_{411} + 2s_{511} + s_{611} + s_{711}\\
&\qquad\qquad\qquad   + s_{44} + s_{33} + s_{43}  + s_{53} + s_{63}  + s_{22} + 2s_{42} + 2s_{52} + s_{62} + s_{72} + s_{82}\\
&\qquad\qquad\qquad   + 2s_{41} + s_{51} + s_{61} + 2s_{71} + s_{81} + s_{91} + s_{(10,1)} + s_{6} + s_{9} + s_{(12)})\otimes e_{33}\\
& \qquad +
(2s_{221} + 3s_{321} + s_{421} + 6s_{111} + 8s_{211} + 8s_{311} + 4s_{411} + 3s_{511} + s_{611} + 2s_{33} + 3s_{43} + s_{53} \\
&\qquad\qquad\qquad  + 4s_{22} + 8s_{32} + 8s_{42} + 5s_{52} + 4s_{62} + s_{72} 
+ 6s_{21} + 10s_{31} + 12s_{41} + 12s_{51} + 10s_{61} + 5s_{71} + 4s_{81}\\
&\qquad\qquad\qquad    + 6s_{4} + 4s_{5} + 4s_{6} + 4s_{7} + 6s_{8} + 2s_{9} + s_{91} + 3s_{(10)} + s_{(11)})\otimes e_{321}\\
%& \qquad \\
& \qquad +
(s_{221} + s_{111} + 2s_{211} + s_{311} + s_{411} + 2s_{22} + s_{32} + s_{42} + s_{52}\\
&\qquad\qquad\qquad    + 2s_{21} + 2s_{31} + 2s_{41} + 2s_{51} + s_{61} + s_{71} + s_{3} + 2s_{5} + s_{7} + s_{9})\otimes e_{222}\\
& \qquad +
(3s_{22} + 2s_{32} + s_{42} + 6s_{11} + 9s_{21}  7s_{31} + 5s_{41} + 2s_{51} + s_{61}
\\&\qquad\qquad\qquad    
 + 6s_{2} + 6s_{3} + 6s_{4} + 4s_{5} + 5s_{6} + 2s_{7} + s_{8})\otimes e_{2211}\\
& \qquad +
(5s_{1} + 4s_{2} + 3s_{3} + 2s_{4} + s_{5})\otimes e_{21111} +
1\otimes e_{111111}.
\end{aligned}$
\end{enumerate}
}
 
 \goodbreak
%%%%%%%%%%%%%%%%%%%%%%
\noindent{\bf Lenght components of $\A_7$.}

{\footnotesize 
\begin{enumerate}\itemsep=3pt\setlength{\itemindent}{-.3in}
\item[]  $ \smash{\A_7^{(1)}}=s_{(21)}$, \qquad  $  \smash{A_7^{(6)}}= s_{111111}$;
\item[]  $\begin{aligned}
	& \smash{\A_7^{(2)}}= s_{77} + (s_{76}+ s_{86}   + s_{96})  + (s_{75} +  s_{85} + s_{95} + s_{(10,5)} + s_{(11,5)})   \\
		&\qquad  + (s_{74} + s_{84}  + 2s_{94} + 2s_{(10,4)}   + 2s_{(11,4)}   + s_{(12,4)} + s_{(13,4)}) \\
		&\qquad  + (s_{93}  + s_{(10,3)} + 2s_{(11,3)} + 2s_{(12,3)} +  2s_{(13,3)}  + s_{(14,3)}+ s_{(15,3)}) \\
		&\qquad  +  (s_{(11,2)}  + s_{(12,2)} + 2s_{(13,2)} + s_{(14,2)}  + 2s_{(15,2)}  + s_{(16,2)}  + s_{(17,2)})\\
		&\qquad  +( s_{(15,1)}  + s_{(16,1)} + s_{(17,1)} + s_{(18,1)} + s_{(19,1)});
    \end{aligned}$
\item[]  $\begin{aligned}
	& \smash{\A_7^{(3)}}=     s_{443} + s_{633}  + (s_{442} + s_{542} + s_{642}+ s_{742})   
	 %\\     &&\qquad      
              + (s_{532} + s_{632}  + s_{732} + s_{832}+ s_{932})   \\
             &\qquad   + (s_{522}  +  s_{722} + s_{822} + s_{922}  + s_{(11,2,2)})
      %       \\ &&\qquad   
             + s_{661} + (s_{651}  + s_{751} + s_{851})  \\
             &\qquad   + (s_{441}  + s_{541}+ 2s_{641} + 3s_{741} + 3s_{841} + 2s_{941}  + s_{(10,4,1)}) \\
             &\qquad  + (s_{631} + 2s_{731} + 3s_{831} + 3s_{931} + 3s_{(10,3,1)} + 2s_{(11,3,1)} + s_{(12,3,1)})\\
             &\qquad   + (s_{721} + 2s_{821} + 2s_{921}  + 3s_{(10,2,1)}  + 3s_{(11,2,1)} 
        %     \\ &&\qquad \qquad\qquad 
             + 2s_{(12,2,1)}  +2s_{(13,2,1)}+ s_{(14,2,1)})  \\
             &\qquad   + (s_{(10,1,1)}  + s_{(11,1,1)} + 2s_{(12,1,1)} + 2s_{(13,1,1)} 
            % \\ &&\qquad \qquad\qquad 
             + 2s_{(14,1,1)}  + s_{(15,1,1)} + s_{(16,1,1)});
   \end{aligned}$
\item[]  $\begin{aligned}
	&\smash{\A_7^{(4)}}=  (s_{4411} + s_{5411}  +s_{6411})+ (s_{4311} + s_{5311}  + 2s_{6311}+ s_{7311}+ s_{8311})\\
              &\qquad   +(s_{4211} +  s_{5211}  + 2s_{6211} + s_{7211}   + 2s_{8211} + s_{9211} + s_{(10,2,1,1)})  \\
              &\qquad   + (s_{6111} + s_{7111} + 2s_{8111}+ 2s_{9111} + 2s_{(10,1,1,1)}  + s_{(11,1,1,1)} + s_{(12,1,1,1)});
   \end{aligned}$
\item[] $ \smash{\A_7^{(5)}}= s_{31111} + s_{41111} + s_{51111} + s_{61111} + s_{71111}$.
 \end{enumerate}
}

\goodbreak
 
%%%%%%%%%%%%%%%%%%%%%%%%%%%%%%%%%%%%%%%%%%%%
\renewcommand{\refname}{\bleu{References}}

\end{document}